\documentclass[11pt]{article}
\usepackage{mathrsfs}
\usepackage{amsmath}
\usepackage{amssymb}
\usepackage{graphicx}
\usepackage{epic}
\usepackage{color}

\usepackage{pst-poly}  
\usepackage{pst-plot}  
\usepackage{pst-poly}  

\renewcommand{\paragraph}{\roman{paragraph}}

\usepackage{bm}
\usepackage{hyperref}
\usepackage{tikz}
\usetikzlibrary{arrows,shapes,positioning}
\usetikzlibrary{decorations.markings}
\tikzstyle arrowstyle=[scale=1]
\tikzstyle directed=[postaction={decorate,decoration={markings, mark=at position .65 with {\arrow[arrowstyle]{stealth}}}}]
\tikzstyle reverse directed=[postaction={decorate,decoration={markings, mark=at position .65 with {\arrowreversed[arrowstyle]{stealth};}}}]

\newtheorem{theorem}{Theorem}[section]
\newtheorem{corollary}[theorem]{Corollary}
\newtheorem{definition}[theorem]{Definition}

\newtheorem{lemma}[theorem]{Lemma}

\newtheorem{proposition}[theorem]{Proposition}
\newtheorem{remark}[theorem]{Remark}
\newenvironment{proof}{\noindent {\bf Proof.}}{\rule{3mm}{3mm}\par\medskip}

\newtheorem{observation}[theorem]{Observation}

\begin{document}

\title{Hermitian adjacency matrices of mixed multigraphs \thanks{This work was supported by National Science Foundation of China (Grant No.~11901525, 12100557)
and by Zhejiang Provincial Natural Science Foundation of China (Grant No.~LQ21A010004).
}}

\author{Bo-Jun Yuan \ \ \ Shaowei Sun \ \ \  Dijian Wang
\thanks{Corresponding
author.
E-mail address: dijianwang0121@163.com (D.-J. Wang), ybjmath@163.com (B.-J. Yuan), sunshaowei2009@126.com (S.-W. Sun).}   \\
{\small  \it School of Science, Zhejiang University of Science and Technology, Hangzhou~310023, PR China}
}


\date{}
\maketitle

\begin{abstract}
A mixed multigraph  is obtained from an undirected  multigraph by orienting a subset of its edges.
In this paper, we study  a new Hermitian matrix representation of mixed multigraphs,
give an introduction to cospectral operations on mixed multigraphs, and characterize switching equivalent
mixed multigraphs in terms of fundamental cycle basis. As an application, an upper bound of cospectral classes of mixed multigraphs with the same underlying graph is obtained.
\end{abstract}


\section{Introduction}
All graphs considered in this paper may contain multiple edges but no loops.
A {\it mixed multigraph} $M$ is obtained from an undirected  multigraph $G$ by orienting
a subset of its edges, where $G$ is called the {\it underlying graph} of $M$, written as $G(M)$.
Formally, a mixed multigraph is an ordered triple $(V, E, A)$, where $V$ is the vertex set, $E$ is the undirected edge set and $A$ is the directed  edge (arc) set.
We shall denote an edge  with endpoints $u$ and $v$ by $\{u, v\}$ if it is undirected, $(u, v)$ if it is directed from the {\it initial vertex} $u$ to the {\it terminal vertex} $v$, and $uv$ if it is either undirected or directed from $u$ to $v$.
If $uv\in A$ and $vu\in A$, the pair $\{uv, vu\}$  of these oppositely directed arcs is referred to as a {\it digon} of $M$, denoted by $\{u,v\}$.
A mixed multigraph can be viewed as a directed graph (henceforth abbreviated as {\it digraph}) if we view each undirected edge as a digon.

One basic problem in spectral graph theory is investigating the relation  between the eigenvalues of some matrix representation of a graph and its structural properties.
In the current setting, it is natural to consider Hermitian matrices as the adjacency matrices of mixed graphs, in which
two conjugate complex numbers are used to reflect the directed information of an edge. Based on this,
Liu and Li \cite{li} and independently by Guo and Mohar \cite{mohar1} introduced the Hermitian adjacency matrix of the first kind for simple mixed graphs, in which
the $uv$-entry is $1$ if $\{u, v\}$ is undirected, is $\mathsf{i}$ $(=\sqrt{-1})$ if $(u,v)$ is an arc, is $-\mathsf{i}$ if $(v,u)$ is an arc, and is $0$ otherwise.
In recent years, the Hermitian spectrum of simple mixed graphs has been the subject of several publications. One can refer to
the literature \cite{Chen1, Chen2, Greaves, Hu, Tian, Wang, Wissing, Yu} and the references therein.
However, it seems that this kind of Hermitian matrix with entries from \{0, 1, $\pm\mathsf{i}$\} is more valid for simple mixed graphs.
As such, the  entries of Hermitian matrices of mixed multigraphs (or digraphs) may need to be restricted on a bigger
algebra system.  Mohar \cite{mohar00} provided an interesting new choice as follows.

Let $D$ be a digraph with vertex set $V$ and arc set $A$.
Let $\alpha=a+b\mathsf{i}$ be a complex number satisfying $|\alpha|=1$ and $a\geq 0,$ and let $\bar{\alpha}=a-b\mathsf{i}$ be its conjugate.
Mohar \cite{mohar00} introduced the {\it Hermitian adjacency matrix} $N^\alpha(D)=[N^\alpha_{uv}]\in \mathbb{C}^{V \times V}$ of $D$ with $$N^\alpha_{uv}=e(u, v)\alpha+e(v, u)\bar{\alpha},$$
where $e(x, y)$ denotes the number of arcs directed from $x$ to $y.$

If we take $\alpha=\mathsf{i} =\sqrt{-1}$, then the corresponding entries of $N^\alpha(D)$ are the Gaussian integers.
In the current setting, digraphs are maybe not in bijective correspondence with their matrices.
For instance, the matrix of digraph consisting of a digon is a null matrix of order 2,  which is the same as the matrix of digraph consisting of two isolated vertices.
Thus, in a sense, taking  the primitive sixth root of unity $\omega=\frac{1}{2}+\frac{\sqrt{3}}{2}\mathsf{i}$ as $\alpha$ may be the most natural choice to characterize the interesting combinatorial structure of a digraph (or a mixed multigraph) by its spectrum.
Note that a digon has the contribution 1 to the corresponding entries in $N^\omega$  because $\omega+\overline{\omega}=1$, which can be naturally considered as the contribution of an undirected edge if we view a digon
 as an undirected edge.
 In addition,  the entries of
Hermitian adjacency matrix  $N^\omega$  are Eisenstein integers which play an important role in formal theory of linear independence, see \cite{Chun, Whittle}.
It also seems to be associated with Quantum Field Theory from theoretical physics \cite{Kal}.

Let us come back to the matrix representation of mixed multigraphs. We emphasize here that since digons behave like undirected edges, for our purposes throughout this paper, we may assume that there are no digons in mixed multigraphs.
Equivalently, for any two adjacent vertices $u$ and $v$, the set of edges between $u$ and $v$ consists of  undirected edges and
arcs with the same direction.
In this way, the {\it Hermitian adjacency matrix} of a mixed multigraph $M=(V, E, A)$ is defined as
$N^\omega(M)=[N^\omega_{uv}]$ with
$$N^\omega_{uv}= \left\{
\begin{array}{ll}
e\{u,v\} & {\rm if \ there \  are \ only \ undirected \ edges \ between} \ u \ {\rm and} \ v;\\
e\{u,v\}+e(u,v)\omega & {\rm if \ there \  are \ arcs \ directed \ from} \ u \ {\rm to} \ v;\\
e\{u,v\}+e(v,u)\overline{\omega} & {\rm if \ there \  are \ arcs \ directed \ from} \ v \ {\rm to} \ u;\\
0 & {\rm otherwise,}
\end{array}
\right.
$$
where $e\{x, y\}$ denotes the number of undirected edges between vertices $x$ and $y$ and $e(x, y)$ denotes the number of arcs directed from $x$ to $y.$ 

Hereafter, we shall briefly denote $N(M)$ by the Hermitian adjacency matrix of a mixed multigraph $M$ instead of the notation $N^\omega(M).$
Obviously, $N(M)$ is a Hermitian matrix, 
and thus is diagonalizable with real eigenvalues.
The {\it eigenvalues}, the {\it spectrum}, and the {\it characteristic polynomial} $\mathnormal{\Phi}(M, \lambda)$ of
a mixed multigraph $M$ always refer to those of its Hermitian adjacency matrix $N(M)$.
Two mixed multigraphs $M$ and $M'$ are called {\it cospectral} if they have the same spectrum, and they
are called {\it antispectral} if for every  eigenvalue $\lambda$ of $M$  of multiplicity $\ell$, $-\lambda$ is an eigenvalue of $M'$ with the same multiplicity $\ell$.

The motivation for this paper is to  study  a new matrix representation of mixed multigraphs. In Section 2, we give the expression of characteristic polynomials of mixed multigraphs. In Section 3, we introduce  some
graphic operations which can preserve the spectrum of mixed multigraphs via its new Hermitian adjacency matrix. In Section 4, we characterize switching equivalent
mixed multigraphs in terms of fundamental cycle basis. As an application, an upper bound of cospectral classes of mixed multigraphs with the same underlying graph is obtained.

\section{Preliminaries}
\subsection{Characteristic polynomials of mixed multigraphs}
A mixed multigraph is called {\it simple} if its underlying graph is a simple graph.
It is said to be a {\it mixed walk}, {\it mixed cycle}, etc., if its underlying graph
is a walk, cycle, etc., respectively. 

Let $M=(V,E,A)$ be a connected mixed multigraph and let $N(M)=[N_{uv}]$ be its Hermitian adjacency matrix.
For  a mixed walk $W = v_1e_{1}v_2e_{2}v_3\cdots v_{s-1}e_{s-1}v_s$ in $M$,
the {\it weight} of $W$ in $M$ is defined as
$$wt(W)=N_{e_{1}}N_{e_{2}}\cdots N_{e_{s-1}},\eqno(2.1)$$
where, for each $i=1,2, \ldots, s-1$, $N_{e_{i}}=1$  if $e_{i}\in E$, $N_{e_{i}}=\omega$  if $e_{i}\in A$ is directed from $v_i$ to $v_{i+1}$, and $N_{e_{i}}=\overline{\omega}$  if $e_{i}\in A$ is directed from $v_{i+1}$ to $v_i$.
For a mixed cycle $C$, we need to assign a direction for $C$ before calculating its weight. Since $N(M)$ is Hermitian, if for one direction the  weight of a mixed cycle  is $\beta$, then for
the reversed direction its weight is $\bar{\beta}$  (the conjugate number of $\beta$).
For a given cycle direction of $C$,
assume that $f$ ($b$, resp.) is the number of forward (backward, resp.) arcs in $C$.
Now we express $\omega=\frac{1}{2}+\frac{\sqrt{3}}{2}\mathsf{i}$ in terms of
trigonometric function, that is, $\omega=\cos\theta+\mathsf{i}\sin\theta,$ where $\theta=\pi/3.$ In doing so, we may obtain a concise
expression for the characteristic polynomial of $M$ below.
We define the {\it value} of $C$ to be $wt(C)+\overline{wt(C)}$, denoted by $\nu(C)$. It follows that
\begin{eqnarray*}
\nu(C)&=&wt(C)+\overline{wt(C)} =\omega^f\overline{\omega}^b+\overline{\omega}^f\omega^b
=2{\rm Re}(\omega^f\overline{\omega}^b)\\
&=&2{\rm Re}\left((\cos(f\theta)+\mathsf{i}\sin(f\theta))\cdot(\cos(b\theta)-\mathsf{i}\sin(b\theta))\right)\\
&=&2\cos((f-b)\theta)=2\cos((f-b)\pi/3).
\end{eqnarray*}
Since $\nu(C)=2\cos((f-b)\theta)=2\cos((b-f)\theta)$, the value of an arbitrary mixed
cycle $C$ is independent of the chosen  direction of $C$.

Note that any mixed cycle $C$ satisfies $\nu(C)\neq 0$ as $f$  and $b$ in $C$ are non-negative integers (see Observation \ref{obser0}).
A subgraph of $M$ whose each component is an (undirected or directed) edge or a mixed cycle (including the mixed cycle of length $2$) is referred to as a {\it Sachs subgraph} of $M$.
We emphasize here that a single edge can not be enumerated as a mixed cycle of length $2$ when we consider the computation of characteristic polynomial of a mixed multigraph.

The following theorem can be considered as an  analogue of  Sachs' Coefficient Theorem \cite[Page 32]{cve}, and its corresponding version for simple mixed graphs appears in \cite[Theorem 2.6]{li1}.
\begin{theorem}\label{cha}
Let $M$ be a mixed multigraph of order $n$ with  the Hermitian  characteristic polynomial $\mathnormal{\Phi}(M, \lambda)=\sum_{i=0}^nc_i\lambda^{n-i}$. 
Then 
$$c_i=\sum_{S}(-1)^{r(S)}\prod_{C\subset S}\nu(C),$$
where the sum runs over all Sachs subgraphs $S$ of order $i$ in $M$, and $r(S)$ is the number of connected  components in  $S$.
\end{theorem}

\subsection{Applications of Theorem \ref{cha}}
As an application of Theorem \ref{cha}, we determine which mixed multigraphs can be  antispectral with their underlying graphs.
We start this subsection with the following observation.

\begin{observation} \label{obser0}
Let $C$ be a mixed cycle, and let $f$ $($$b$, resp.$)$ be the number of forward $($backward, resp.$)$ arcs in $C$ with arbitrary cycle direction.
By the definition of $\nu(C),$ we have
$$\nu(C)=2 \Leftrightarrow |f-b|\equiv 0 ({\rm mod} \ 6)\Leftrightarrow wt(C)=1;$$
$$\nu(C)=-2 \Leftrightarrow  |f-b|\equiv 3 ({\rm mod} \ 6)\Leftrightarrow wt(C)=-1;$$
$$\nu(C)=1 \Leftrightarrow |f-b|\equiv 1 ({\rm mod} \ 6) ~or~ |f-b|\equiv 5 ({\rm mod} \ 6)\Leftrightarrow wt(C)=\omega ~or~ wt(C)=\omega^5;$$
$$\nu(C)=-1 \Leftrightarrow |f-b|\equiv 2 ({\rm mod} \ 6) ~or~ |f-b|\equiv 4 ({\rm mod} \ 6) \Leftrightarrow wt(C)=\omega^2 ~or~ wt(C)=\omega^4.$$
\end{observation}


\begin{theorem} \label{antisun0}  Let $M$ be a mixed multigraph of order $n$. Then, $M$ is antispectral to its underlying graph if and only if every even mixed cycle of $M$ has weight $1$
and every odd mixed cycle of $M$ has weight $-1$.
\end{theorem}
\begin{proof}
Write $G=G(M)$. Let $\mathnormal{\Phi}(G, \lambda)=\sum_{i=0}^nc_i\lambda^{n-i}$ and $\mathnormal{\Phi}(M, \lambda)=\sum_{i=0}^nc'_i\lambda^{n-i}$.
Then, $M$ and $G$ are antispectral if and only if $c_j=c'_j$ (for even $j$) and $c_j=-c'_j$ (for odd $j$).

The sufficiency is easily verified by Theorem \ref{cha} and Observation \ref{obser0}, and we only prove the necessity.
Assume to the contrary. There exists either some even cycles with weight no equal to $1$ or some odd cycles with weight not equal to $-1$.
Let $k$ be the minimum length among these cycles.

{\bf Case 1:} {\it $k$ is even.}  In order to yield a contradiction, we need to calculate $c_k$ and $c'_k$ using Theorem \ref{cha}.
Let $S_1$ and $S_1'$ be two corresponding Sachs subgraphs of order $k$ consisting of $k/2$ isolated edges of $G$ and $M$, respectively.
Clearly, $$(-1)^{r(S_1)}\prod_{C \subset S_1}\nu(C)=(-1)^{r(S'_1)}\prod_{C\subset S'_1} \nu(C)=(-1)^{k/2}.$$
Let $S_2$ and $S_2'$ be two corresponding Sachs subgraphs of order $k$ containing at least one cycle with length less than $k$ of $G$  and $M$, respectively.
Notice that all even (odd, resp.) mixed cycles of length less than $k$ in $M$ have weight $1$ ($-1$, resp.) and the fact that every Sachs subgraph $S_2$ (and $S_2'$) contains even number of odd mixed cycles.
It follows that
$$(-1)^{r(S_2)} \prod_{C\subset S_2}\nu(C)=(-1)^{r(S'_2)} \prod_{C\subset S'_2}\nu(C).$$
Let $S_3$ and $S_3'$ be two corresponding cycles of $G$ and $M$ with length $k$, respectively. By Theorem \ref{cha}, we deduce that
$$c_{k}-c'_{k}
=\sum_{S_3}(-1)\times 2 -\sum_{S_3'}(-1)\times \nu(S_3').$$
Then, the fact $c_{k}=c'_{k}$ forces that for any $S_3'$, its weight is $1$ by Observation \ref{obser0}, a contradiction.

{\bf Case 2:} {\it $k$ is odd.} The proof is similar to Case 1, which is left to the reader.
\end{proof}

\begin{remark} \label{cospecu}
Using the similar proof in Theorem \ref{antisun0}, we can also obtain that a mixed multigraph  and its underlying graph are cospectral if and only if all mixed cycles in the mixed multigraph  have weight $1$.
\end{remark}

The following two consequences immediately follow from Theorem \ref{cha} since the value of any cycle is invariant if we change the orientation of a bridge.

\begin{corollary}\label{cut}
If $uv$ is a bridge of a mixed multigraph, then the spectrum is invariant when we change the direction of $uv$.
\end{corollary}

\begin{proposition}\label{for}
All mixed forests with the same underlying graph are cospectral.
\end{proposition}

\section{Cospectral operations on mixed multigraphs}
From the perspective of spectral analysis, there are many cospectral mixed graphs with the same underlying graph.
Thus, researchers turn to consider the cospectral problem on classes of ``switching equivalent" mixed graphs rather than on individual mixed graphs~\cite{li1,mohar2,Wang,Wangy}.
Motivated by this, in this section, we will present operations that preserve the spectrum of mixed multigraphs.
For similar cospectral operations on simple mixed graphs, one can refer to Li and Yu \cite{li1}. Firstly, we describe our operations in matrix language below.
\begin{definition} \label{sixm}
Let $M'$ and $M''$ be two mixed multigraphs with the same underlying graph.
A {\it three-way switching} is the operation of changing $M'$ into $M''$, if there exists a diagonal matrix $D$ with
$D_{vv}\in \{\omega^i| 0\leq i\leq 5\}$ such that $$N(M'')=D^{-1}N(M')D,$$ or equivalently, $$N_{e''_{uv}}=D^{-1}_{uu}N_{e'_{uv}}D_{vv}$$ for any corresponding edges
$e'_{uv}$ and $e''_{uv}$.
\end{definition}

Definition \ref{sixm} implies that if $M''$ is obtained from $M'$ by a sequence of three-way switchings, then $M''$ can be obtained from $M'$ by a single three-way switching, and the operation of three-way switching
is invertible, that is, $M'$ can be also obtained from $M''$ by a three-way switching. The following proposition is also immediately obtained by Definition \ref{sixm}.

\begin{proposition} \label{sixcos}
The three-way switching on any mixed multigraph can give a cospectral mixed multigraph.
\end{proposition}

Given two mixed multigraphs with the same underlying graph, even they are cospectral, it is not easy to determine
whether one of them can be obtained from the other by a three-way switching or not.
Next, we provide an equivalent condition to understand the operation from a structural point of view.

Let $M'$ and $M''$ be two mixed multigraphs with the same underlying graph.
Taking any mixed cycle in $M'$, we first assign a direction for it and  we can obtain the corresponding mixed cycle in $M''$ naturally.
Hereafter, discussing two mixed multigraphs with the same underlying graph, we always assume that  all corresponding mixed cycles of them have the same cycle direction.
\begin{theorem} \label{judge}
Let $M'$ and $M''$ be two connected mixed multigraphs with the same underlying graph $G$ on $n$ vertices. Then one mixed multigraph
can be obtained from the other by a three-way switching if and only if all corresponding mixed cycles of them have the same weights.
\end{theorem}

To prove Theorem \ref{judge}, we need the following lemma.

\begin{lemma} \label{mtree}
Let $M'$  and $M''$ be two mixed trees with the same underlying graph. Then one of them  can be  obtained from the other by a three-way switching.
\end{lemma}
\begin{proof}
Let $N(M')=[N'_{uv}]$ and $N(M'')=[N''_{uv}]$.
Now we define a function $f: V(M')\rightarrow \mathfrak{W}=\{\omega^i| 0\leq i\leq 5\}$ by the following procedure. First, pick an arbitrary vertex $v$ of $M'$ and set $f(v) =1$, then we expand the definition of $f$ through adjacency relation.
For defined vertex  $v$ of $M'$ and its undefined neighbor $u,$ let $f^{-1}(u)N'_{uv}f(v) =N''_{uv}$, i.e., $f(u) =f(v)N'_{uv}{N''_{uv}}^{-1}$. 
The above process will terminate as $M'$ is a finite graph. Let $D=[D_{vv}]$ be the diagonal matrix with
$D_{vv}=f(v).$ It is easy to see that $N(M'')=D^{-1}N(M')D$.
The  result follows from the definition of three-way switching.
\end{proof}

\noindent\textbf{Proof of Theorem \ref{judge}.} Denote by $N(M')=[N'_{uv}]$ and $N(M'')=[N''_{uv}]$.

Necessity: By Definition \ref{sixm}, there exists a diagonal matrix $D=[D_{vv}]$ such that
$N''_{e}=D^{-1}_{uu}N'_{e}D_{vv}$ for each edge  $e=uv\in E(G)$. Denote a mixed cycle in $M'$ by $C' = v_1e'_{1}v_2e'_{2}v_3\cdots v_{s}e'_{s}v_1$ and its
corresponding mixed cycle in $M''$ by $C'' = v_1e''_{1}v_2e''_{2}v_3\cdots v_{s}e''_{s}v_1$. Then
\begin{align*}
wt(C'')=& N''_{e''_{1}}N''_{e''_{2}}\cdots N''_{e''_{s}}\\
=& D^{-1}_{v_1v_1}N'_{e'_{1}}D_{v_2v_2}D^{-1}_{v_2v_2}N'_{e'_{2}}D_{v_3v_3}\cdots D^{-1}_{v_sv_s}N'_{e'_{s}}D_{v_1v_1}\\
=& N'_{e'_{1}}N'_{e'_{2}}\cdots N'_{e'_{s}}=wt(C').
\end{align*}

Now we prove the sufficiency by induction on the dimension of cycle space of $M'$, say $d$.
If $d=0$, then $M'$  and $M''$  are two mixed trees with the same underlying graph, and
the assertion follows from Lemma \ref{mtree}.
We may suppose that $d\geq 1$ and the assertion holds for all connected mixed multigraphs  with dimension less than $d$.
Let $M'$ be a connected mixed multigraph  with dimension $d$.
Take a mixed cycle $C' = v_1e'_{1}v_2e'_{2}v_3\cdots v_{s}e'_{s}v_1$ in $M'$ and a mixed edge $e'_{1}\in E(C')$ and
denote the corresponding mixed cycle and mixed edge in $M''$ by $C''$ and $e''_{1}$, respectively.
By induction hypothesis, there exists a diagonal matrix $D=[D_{vv}]$ such that
$$N''_{e''_{uv}}=D^{-1}_{uu}N'_{e'_{uv}}D_{vv}$$
for any corresponding edges~$e'_{uv}$ and $e''_{uv}$ of $M'-e'_{1}$ and $M''-e''_{1}$, respectively.
To complete the proof, it suffices to prove
$N''_{e''_{1}}=D^{-1}_{v_1v_1}N'_{e'_{1}}D_{v_2v_2}$. It follows from $wt(C')=wt(C'')$ that
\begin{align*}
N'_{e'_{1}}N'_{e'_{2}}\cdots N'_{e'_{s}}=& N''_{e''_{1}}N''_{e''_{2}}\cdots N''_{e''_{s}}\\
=& N''_{e''_{1}}D^{-1}_{v_2v_2}N'_{e'_{2}}D_{v_3v_3}\cdots D^{-1}_{v_sv_s}N'_{e'_{s}}D_{v_1v_1}.
\end{align*}
The result follows. $\blacksquare$

A mixed cycle $C$ is called {\it positive} if $wt(C)=1.$ 
A mixed multigraph is said to be {\it positive} if each its mixed cycle is positive.
Theorem \ref{judge} easily implies the following corollary.
\begin{corollary} \label{cosunder}
Let $M$ be a mixed multigraph. Then $M$  can be obtained from its underlying graph by a three-way switching if and only if $M$ is positive.
\end{corollary}

Now we describe the operation of a three-way switching in the graphic structural language. The similar interpretation was given by Li and Yu \cite{li1} for simple mixed graphs.
Assume that there is a partition of the vertex set of a mixed multigraph $M$ into six (possibly empty) sets, $V(M) = V_1 \cup V_\omega \cup V_{\omega^2}\cup V_{\omega^3}\cup V_{\omega^4}\cup V_{\omega^5}$, and it is called
{\it admissible} if the following conditions hold:

(1) For each $j\in \{\omega^i| 0\leq i\leq 5\},$ there are no arcs of $M$ directed from $V_{j}$ to $V_{\omega j}.$

(2) For each $j\in \{\omega^i| 0\leq i\leq 5\},$ there are neither undirected nor directed edges of $M$ with one endpoint in $V_{j}$ and the other endpoint in $V_{\omega^3 j}.$

(3) For each $j\in \{\omega^i| 0\leq i\leq 5\},$ all edges of $M$ with one endpoint in $V_{j}$ and the other endpoint in $V_{\omega^4 j}$ are arcs directed from $V_{j}$ to $V_{\omega^4 j}.$ 

A {\it three-way switching} with respect to an  admissible partition  $V(M) = V_1 \cup V_\omega \cup V_{\omega^2}\cup V_{\omega^3}\cup V_{\omega^4}\cup V_{\omega^5}$ is the operation of changing $M$ into the mixed multigraph $M'$ by making the following transformations, see Fig. \ref{f1}.
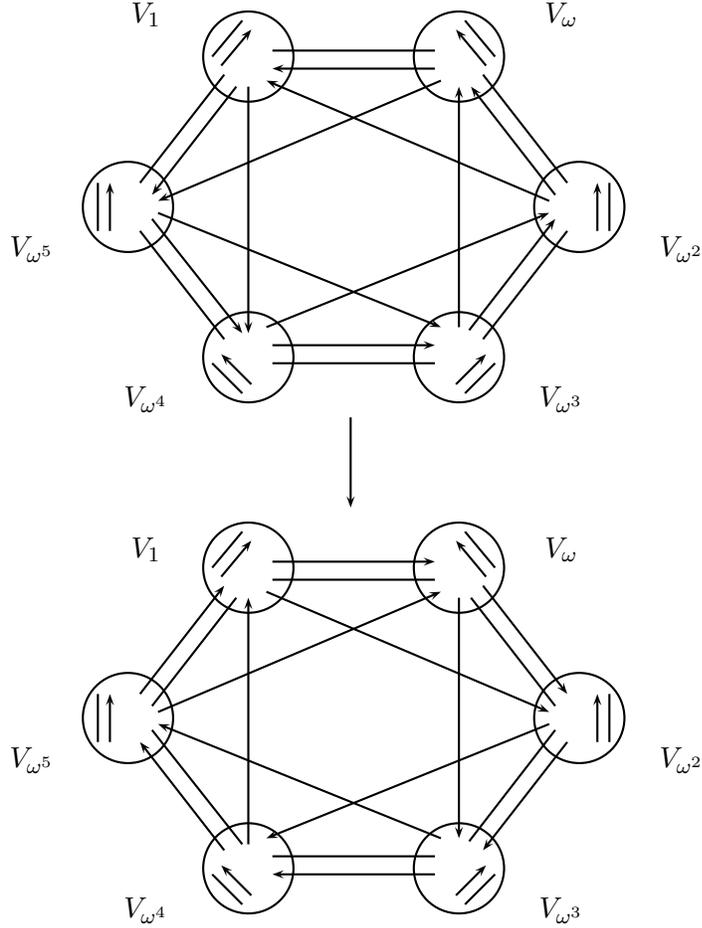
\begin{figure}[ht]
\psset{unit=.8}
\begin{center}
\begin{pspicture}(-11,-8)(0,8)
\pnode(-6.3,1){a}  \pnode(-6.3,-.5){b}  \ncline{->}{a}{b}

\rput(-.8,3.8){$V_{\omega^2}$} \rput(-11.6,3.8){$V_{\omega^5}$}

\cnode(-8,2){.6cm}{1} \rput(-9.7,1.3){$V_{\omega^4}$}
\cnode(-8,7){.6cm}{2}\rput(-9.7,7.7){$V_{1}$}
\cnode(-4.5,2){.6cm}{3}\rput(-2.8,1.3){$V_{\omega^3}$}
\cnode(-4.5,7){.6cm}{4}\rput(-2.8,7.7){$V_{\omega}$}
\cnode(-10,4.5){.6cm}{A}
\cnode(-2.5,4.5){.6cm}{B}

\pnode(-10.3,4.1){A1}  \pnode(-10.3,4.9){A2}  \ncline{<-}{A2}{A1}
\pnode(-10.5,4.1){A3}  \pnode(-10.5,4.9){A4}  \ncline{A4}{A3}
\pnode(-9.6,4.7){A5}  \pnode(-8.2,6.5){C1}  \ncline{<-}{A5}{C1}
\pnode(-9.8,4.9){A6}  \pnode(-8.4,6.7){C2}  \ncline{A6}{C2}
\pnode(-9.8,4.1){A8}  \pnode(-8.4,2.3){C4}  \ncline{A8}{C4}
\pnode(-9.6,4.3){A7}  \pnode(-8.1,2.4){C3}  \ncline{->}{A7}{C3}

\pnode(-2.2,4.1){B1}  \pnode(-2.2,4.9){B2}  \ncline{<-}{B2}{B1}
\pnode(-2,4.1){B3}  \pnode(-2,4.9){B4}  \ncline{B4}{B3}
\pnode(-4.33,2.45){D1}  \pnode(-2.9,4.3){B3}  \ncline{->}{D1}{B3}
\pnode(-4.1,2.3){D2}  \pnode(-2.7,4.1){B4}  \ncline{D2}{B4}
\pnode(-4.3,6.5){B5}  \pnode(-2.9,4.7){B7}  \ncline{->}{B7}{B5}
\pnode(-4.1,6.7){B6}  \pnode(-2.7,4.9){B8}  \ncline{B6}{B8}

\pnode(-8.6,1.9){11}  \pnode(-8.1,1.4){12}  \ncline{11}{12}
\pnode(-8.45,2.05){13}  \pnode(-7.95,1.55){14}  \ncline{<-}{13}{14}
\pnode(-8,2.4){19} \pnode(-8,6.5){29}  \ncline{<-}{19}{29}

\pnode(-7.6,1.9){15}  \pnode(-4.9,1.9){35} \ncline{15}{35}
\pnode(-7.6,2.2){16} \pnode(-4.9,2.2){36}  \ncline{<-}{36}{16}
\pnode(-7.6,7.1){25}  \pnode(-4.9,7.1){45} \ncline{25}{45}
\pnode(-7.6,6.8){26} \pnode(-4.9,6.8){46} \ncline{->}{46}{26}

\pnode(-8.6,7){21}  \pnode(-8.1,7.6){22}  \ncline{21}{22}
\pnode(-8.45,6.85){23}  \pnode(-7.95,7.45){24}  \ncline{->}{23}{24}

\pnode(-3.9,7){41}  \pnode(-4.4,7.6){42}  \ncline{41}{42}
\pnode(-4.55,7.45){43}  \pnode(-4.05,6.85){44}  \ncline{->}{44}{43}

\pnode(-3.9,1.9){31}  \pnode(-4.4,1.4){32}  \ncline{31}{32}
\pnode(-4.05,2.05){33}  \pnode(-4.55,1.55){34}  \ncline{<-}{33}{34}

\pnode(-4.5,2.5){X}  \pnode(-4.5,6.5){Y}  \ncline{->}{X}{Y}

\pnode(-9.5,4.6){F1}  \pnode(-4.8,6.6){F2}  \ncline{->}{F2}{F1}
\pnode(-9.5,4.4){F3}  \pnode(-3,4.6){F4}  \pnode(-3,4.4){F5}
\pnode(-7.7,6.6){F6}
\ncline{->}{F4}{F6}
\pnode(-7.7,2.5){F7}\pnode(-4.8,2.5){F8}
\ncline{->}{F7}{F5} \ncline{->}{F3}{F8}

\rput(-.8,-4.7){$V_{\omega^2}$} \rput(-11.6,-4.7){$V_{\omega^5}$}

\cnode(-8,-6.5){.6cm}{110} \rput(-9.7,-7.2){$V_{\omega^4}$}
\cnode(-8,-1.5){.6cm}{210}\rput(-9.7,-1.2){$V_{1}$}
\cnode(-4.5,-6.5){.6cm}{310}\rput(-2.8,-7.2){$V_{\omega^3}$}
\cnode(-4.5,-1.5){.6cm}{410}\rput(-2.8,-1.2){$V_{\omega}$}
\cnode(-10,-4){.6cm}{A}
\cnode(-2.5,-4){.6cm}{B}

\pnode(-10.3,-4.4){A11}  \pnode(-10.3,-3.6){A21}  \ncline{<-}{A21}{A11}
\pnode(-10.5,-4.4){A31}  \pnode(-10.5,-3.6){A41}  \ncline{A41}{A31}
\pnode(-9.6,-3.8){A51}  \pnode(-8.2,-2){C11}  \ncline{A51}{C11}
\pnode(-9.8,-3.6){A61}  \pnode(-8.4,-1.8){C21}  \ncline{->}{A61}{C21}
\pnode(-9.8,-4.4){A81}  \pnode(-8.4,-6.2){C41}  \ncline{<-}{A81}{C41}
\pnode(-9.6,-4.2){A71}  \pnode(-8.1,-6.1){C31}  \ncline{A71}{C31}

\pnode(-2.2,-4.4){B11}  \pnode(-2.2,-3.6){B21}  \ncline{<-}{B21}{B11}
\pnode(-2,-4.4){B31}  \pnode(-2,-3.6){B41}  \ncline{B41}{B31}
\pnode(-4.33,-6.05){D11}  \pnode(-2.9,-4.2){B31}  \ncline{D11}{B31}
\pnode(-4.1,-6.2){D21}  \pnode(-2.7,-4.4){B41}  \ncline{<-}{D21}{B41}
\pnode(-4.3,-2){B51}  \pnode(-2.9,-3.8){B71}  \ncline{B71}{B51}
\pnode(-4.1,-1.8){B61}  \pnode(-2.7,-3.6){B81}  \ncline{->}{B61}{B81}

\pnode(-8.6,-6.6){111}  \pnode(-8.1,-7.1){121}  \ncline{111}{121}
\pnode(-8.45,-6.45){131}  \pnode(-7.95,-6.95){141}  \ncline{<-}{131}{141}
\pnode(-8,-6.1){191} \pnode(-8,-2){291}  \ncline{<-}{291}{191}

\pnode(-7.6,-6.6){151}  \pnode(-4.9,-6.6){351} \ncline{<-}{151}{351}
\pnode(-7.6,-6.3){161} \pnode(-4.9,-6.3){361}  \ncline{361}{161}
\pnode(-7.6,-1.4){251}  \pnode(-4.9,-1.4){451} \ncline{->}{251}{451}
\pnode(-7.6,-1.7){261} \pnode(-4.9,-1.7){461} \ncline{461}{261}

\pnode(-8.6,-1.5){211}  \pnode(-8.1,-.9){221}  \ncline{211}{221}
\pnode(-8.45,-1.65){231}  \pnode(-7.95,-1.05){241}  \ncline{->}{231}{241}

\pnode(-3.9,-1.5){411}  \pnode(-4.4,-.9){421}  \ncline{411}{421}
\pnode(-4.55,-1.05){431}  \pnode(-4.05,-1.65){441}  \ncline{->}{441}{431}

\pnode(-3.9,-6.6){311}  \pnode(-4.4,-7.1){321}  \ncline{311}{321}
\pnode(-4.05,-6.45){331}  \pnode(-4.55,-6.95){341}  \ncline{<-}{331}{341}

\pnode(-4.5,-6){X1}  \pnode(-4.5,-2){Y1}  \ncline{->}{Y1}{X1}

\pnode(-9.5,-3.9){F11}  \pnode(-4.8,-1.9){F21}  \ncline{->}{F11}{F21}
\pnode(-9.5,-4.1){F31}  \pnode(-3,-3.9){F41}  \pnode(-3,-4.1){F51}
\pnode(-7.7,-1.9){F61}
\ncline{->}{F61}{F41}
\pnode(-7.7,-6){F71}\pnode(-4.8,-6){F81}
\ncline{->}{F51}{F71} \ncline{->}{F81}{F31}
\end{pspicture}
\caption{\label{f1}\footnotesize Three-way switching with respect to an admissible partition.}
\end{center}
\end{figure}

(1) For each $j\in \{\omega^i| 0\leq i\leq 5\},$ replacing each undirected edge between $V_{j}$ and $V_{\omega j}$ with a single arc
directed  from $V_{j}$ to $V_{\omega j}.$

(2) For each $j\in \{\omega^i| 0\leq i\leq 5\},$ replacing each arc directed  from $V_{j}$ to $V_{\omega^5 j}$ with a single undirected edge.

(3) For each $j\in \{\omega^i| 0\leq i\leq 5\},$ reversing the direction of all arcs directed  from $V_{j}$ to $V_{\omega^4 j}$.

\begin{theorem} \label{COS}
If a partition $V(M) = V_1 \cup V_\omega \cup V_{\omega^2}\cup V_{\omega^3}\cup V_{\omega^4}\cup V_{\omega^5}$  is admissible, then the mixed multigraph $M'$ obtained from $M$ by a three-way switching is cospectral with $M.$
\end{theorem}
\begin{proof}
Let  $D$ be the diagonal matrix such that
$D_{vv}=j\in  \{\omega^i| 0\leq i\leq 5\}$ if $v\in V_j.$ It is easy to check that $N(M')=D^{-1}N(M)D$.
The result follows.
\end{proof}

\begin{definition} \label{converse}
Let $M$ be a mixed multigraph. The {\it converse of $M$},  denoted by $M^T$, is the mixed multigraph obtained by reversing the direction of all arcs of $M$.
\end{definition}

By Definition \ref{converse}, we see that a mixed multigraph $M$ and its converse $M^T$ are cospectral  since $N(M^T)=N(M)^T=\overline{N(M)}.$
This implies the following result.

\begin{proposition} \label{cconverse}
Let $M$ be a mixed multigraph. If $C'$ and $C''$ are two corresponding mixed cycles in $M$ and $M^T$, respectively, then $wt(C')=\overline{wt(C'')}$.
\end{proposition}

Combining Theorem \ref{judge} with Proposition \ref{cconverse}, we observe that the operation of converse does not belong to three-way
switching. Based on this,  we present the following definition.

\begin{definition} \label{seq}
Two mixed multigraphs $M'$ and $M''$  are said to be {\it switching equivalent} if one can be
obtained from the other by a three-way switching and/or taking a converse.
\end{definition}

Denote by $\mathscr{M}(G)$  the set of all mixed multigraphs whose underlying graph is the undirected  multigraph $G$.
Next, we will show that switching equivalence forms an equivalence relation on $\mathscr{M}(G)$.
We use $\mathfrak{W}=\{\omega^i| 0\leq i\leq 5\}$ to denote the multiplicative group generated by  $\omega$ and use
$\mathfrak{D}$ to denote the set of all diagonal matrices whose diagonal elements belong to $\mathfrak{W}$. For a mixed multigraph $M$, each $D\in \mathfrak{D}$
corresponds to a partition $V(M) = V_1 \cup V_\omega \cup V_{\omega^2}\cup V_{\omega^3}\cup V_{\omega^4}\cup V_{\omega^5}$, where $V_j=\{v\in V(M)|D_{vv}=j\}, j\in \mathfrak{W}$.
\begin{theorem} \label{er}
Let $G$ be an undirected  multigraph. Then  switching equivalence forms an equivalence relation on the set $\mathscr{M}(G).$
\end{theorem}
\begin{proof}
Note that switching equivalent mixed multigraphs have the same underlying graph.
Since $\mathfrak{D}$ is a group for matrix multiplication, the similarity relation with matrices in $\mathfrak{D}$ is symmetric and transitive.
It is easy to see that $\mathfrak{D}$ is closed for complex conjugation and that $N^T=\overline{N}$. If
$M''$ is obtained from $M'$ by a similarity transformation following by taking the converse, then
$$N(M'')=\overline{D^{-1}N(M')D}=\overline{D^{-1}}~\overline{N(M')}~\overline{D}=D~\overline{N(M')}~D^{-1}.$$
It follows that $M''$ is obtained by  a three-way switching from the converse of $M'$. This now easily implies the result.
\end{proof}

\begin{remark} \label{equivalenceclass}
The above proof implies that the switching equivalence class of a mixed multigraph contains all mixed multigraphs obtained from itself or from its converse by a single application of a three-way switching.
\end{remark}

Combining Theorem  \ref{judge},  Proposition \ref{cconverse} with Theorem  \ref{er}, we now provide an equivalent
characterization of switching equivalence between two mixed multigraphs.
\begin{theorem} \label{mai}
Two connected mixed multigraphs with the same underlying graph are switching equivalent if and only if the weights of  all corresponding mixed cycles of them are equal or  conjugate to each other.
\end{theorem}

%
%
%
%
%
%
%
%
%
%
%
%


\section{The number of cospectral classes}

As an application of Section 3, in the last section we study the number of cospectral classes in $\mathscr{M}(G)$, denoted by $n_c(G)$.
It is obvious that cospectral relation forms an equivalence relation on $\mathscr{M}(G)$. Since mixed multigraphs in the same switching equivalence class are cospectral, $n_c(G)\leq n_s(G),$ where $n_s(G)$ denotes the number of switching equivalence classes in $\mathscr{M}(G)$.


Let $G$ be a connected mixed multigraph and $T$ be a spanning tree of $G$.  For any $e\in E(G)\backslash E(T)$, adding $e$ to $T$, we obtain the resultant graph containing the unique cycle, called  the {\it fundamental cycle} of~$G$  with respect to~$T$ and~$e$, denoted by $C_T(e)$. Set $\{C_T(e): e\in E(G)\backslash E(T)\}$ is denoted by $\mathbf{B}_T(G)$ and is said to be the {\it fundamental cycle basis} of~$G$  with respect to~$T$.

Using the notation fundamental cycle basis, the following theorem gives a more brief condition than Theorem \ref{mai} to characterize the switching equivalence  between two mixed multigraphs.
\begin{theorem} \label{multilast0}
Two connected mixed multigraphs with the same underlying graph are switching equivalent if and only if the weights of  all corresponding fundamental cycles of them are equal or  conjugate to each other.
\end{theorem}

To prove Theorem \ref{multilast0}, we need the following lemma.

\begin{lemma} \label{multijudge}
Let $M'$ and $M''$ be two connected mixed multigraphs on $n$ vertices  with the same underlying graph~$G$, and let $T$ be a  spanning tree of $G$. Then one mixed multigraph
can be obtained from the other by a three-way switching if and only if for each fundamental cycle~$C\in \mathbf{B}_T(G)$,~$wt_{M'}(C)=wt_{M''}(C)$.
\end{lemma}
\begin{proof}
The necessity follows from  Theorem \ref{judge}, and we only prove the sufficiency. Denote by~$N(M')=[N'_{uv}],~N(M'')=[N''_{uv}]$.
Lemma~\ref{mtree} implies that there exists a diagonal matrix~$D=[D_{vv}]$  such that for each edge~$e=uv\in E(T)$,
$N''_{e}=D^{-1}_{uu}N'_{e}D_{vv}$. By Definition~\ref{sixm}, it is sufficient to show that for each edge~$f=xy\in E(G)\backslash E(T)$, $N''_f=D^{-1}_{xx}N'_{f}D_{yy}$.~Denote by $C_T(f)= v_1e_{1}v_2\cdots e_{s-1} x f y$, where $x=v_s, y=v_1$.~The condition implies that $wt_{M'}(C_T(f))=wt_{M''}(C_T(f))$, equivalently,
$$ N'_{e_{1}}\cdots N'_{e_{s-1}}N'_{f}
= N''_{e_{1}}\cdots N''_{e_{s-1}}N''_{f}.$$
For any $i=1,2, \ldots, s-1$, noting that $e_i \in E(T)$, we have
$$N''_{e_{i}}=D^{-1}_{v_{i}v_{i}}N'_{e_i}D_{v_{i+1}v_{i+1}},$$
which yields $N''_f=D^{-1}_{v_{s}v_{s}}N'_{f}D_{v_1v_1}$. The result follows.
\end{proof}

\noindent\textbf{Proof of Theorem \ref{multilast0}.}  The necessity follows from  Theorem \ref{mai}, and we only prove the sufficiency.
If for each fundamental cycle~$C\in \mathbf{B}_T(G)$,~$wt_{M'}(C)=wt_{M''}(C)$, then Lemma~\ref{multijudge} implies that the conclusion holds. If for each fundamental cycle~$C\in \mathbf{B}_T(G)$,~$wt_{M'}(C)= \overline{wt_{M''}(C)}$, then, taking the converse of~$M''$, we obtain the resultant graph, denoted by~$(M'')^T.$  By Proposition~\ref{cconverse},
for every fundamental cycle~$C\in \mathbf{B}_T(G)$,~$wt_{(M'')^T}(C)= \overline{wt_{M''}(C)}=wt_{M'}(C).$
The result follows from Remark \ref{equivalenceclass} and Lemma~\ref{multijudge}.
$\blacksquare$

\begin{remark} \label{Remarkend}
Let $M'$ and $M''$ be two connected mixed multigraphs with the same underlying graph~$G$, and let $T$ be a  spanning tree of $G$. Then Theorems~\ref{mai} and~\ref{multilast0} imply that,
$wt_{M'}(C)=wt_{M''}(C)$ $($$wt_{M'}(C)= \overline{wt_{M''}(C)}$$)$  for each cycle~$C$ in  $G$
if and only if $wt_{M'}(C)=wt_{M''}(C)$ $($$wt_{M'}(C)= \overline{wt_{M''}(C)}$$)$ for each fundamental cycle~$C\in \mathbf{B}_T(G)$.
\end{remark}

At the end of this paper, we provide an  upper bound of the number of switching equivalence classes in the set of mixed multigraphs with the same underlying graph, which is naturally an  upper bound of the number of cospectral classes. Similar results for Hermitian adjacency matrix of the first kind can be referred to Wang and Yuan  \cite{Wangy}.

Let $G$ be a  connected~$k$-cyclic undirected multigraph, and let $T$ be a  spanning tree of $G$. Denote
$\mathbf{B}_T(G):=\{C_1,C_2,\cdots,C_k\}$ by the fundamental cycle basis of $G$.
For any $M\in\mathscr{M}(G)$ with the given cycle direction, define the ordered sequence~$\mathfrak{s}(M)=(wt(C_1), wt(C_2),\ldots, wt(C_k))$ as the {\it essential vector} of $M$ which is dependent of $T$ and the given cycle direction. Let $M'\in\mathscr{M}(G)$ and $M''\in\mathscr{M}(G),$ and recall that we always assume that  all corresponding mixed cycles of them have the same cycle direction. Two vectors $\mathfrak{s}(M')$ and $\mathfrak{s}(M'')$ are called {\it conjugate equivalent} if $\mathfrak{s}(M')=\mathfrak{s}(M'')$ or $\mathfrak{s}(M')= \overline{\mathfrak{s}(M'')}.$

Let $\mathcal{S}(G)$ be the set  of essential vectors of all mixed multigraphs in $\mathscr{M}(G)$. It is easy to see that the conjugate equivalence naturally  forms an equivalence relation on $\mathcal{S}(G)$. Further,
Theorem \ref{multilast0} implies that determining $n_s(G)$ is equivalent to determining the value of conjugate equivalence classes in $\mathcal{S}(G)$. Thus, the following upper bound of $n_s(G)$ can be obtained immediately.

\begin{theorem} \label{mai00}
For a connected~$k$-cyclic $(k\geq 1)$ undirected multigraph~$G$, $n_c(G)\leq n_s(G) \leq \frac{6^{k}}{2}+2^{k-1}$.
\end{theorem}
\begin{proof}
Denote by~$a_i$ the number of conjugate equivalence classes in $\mathcal{S}(G)$ whose essential
vectors have exactly $i$ imaginary entries at the same position.
If $i=0$, then the corresponding essential vectors are real vectors and thus~$a_0 \leq 2^k$.
If~$1\leq i \leq k$,  $a_i \leq \frac{1}{2}\binom{k}{i} 2^{k+i}$. Then
$$n_s(G)=\sum_{i=0}^ka_i \leq 2^k+ \sum_{i=1}^k \frac{1}{2}\binom{k}{i} 2^{k+i}= \frac{6^{k}}{2}+2^{k-1}.$$
The result follows.
\end{proof}


\end{document}